\RequirePackage{fix-cm}

\documentclass[]{article}       
\usepackage{amsfonts,enumerate,amsmath,amsthm,authblk}
\usepackage[natbibapa,nodoi]{apacite}
\usepackage{tikz}
\usetikzlibrary{shapes,arrows,fit,calc,positioning,patterns,decorations.pathreplacing}

\newcommand{\J}{\mathcal{J}}

\newtheorem{theorem}{Theorem}
\newtheorem{lemma}{Lemma}
\newtheorem{claim}{Claim}

\begin{document}

\title{New complexity and approximability results for minimizing the total weighted completion time on a single machine subject to non-renewable resource constraints}

\author{P\'eter Gy\"orgyi\footnote{gyorgyi.peter@sztaki.hu} }
\author{Tam\'as Kis\footnote{kis.tamas@sztaki.hu}}
\affil{\small Institute for Computer Science and Control\\ 1111 Budapest, Kende str. 13-17, Hungary}

\maketitle

\begin{abstract}
In this paper we consider single machine scheduling problems with additional non-renewable resource constraints. Examples for non-renewable resources include raw materials, energy, or money.
Usually they have an initial stock and replenishments arrive over time at a-priori known time points and quantities. 
The jobs have some requirements from the resources and a job can only be started if the available quantity from each of the required resources exceeds the requirements of the job. Upon starting a job, it consumes its requirements which decreases the available quantities of the respective non-renewable resources.
There is a broad theoretical and practical background for this class of problems.  Most of the literature concentrate on the makespan, and the maximum lateness objectives.
This paper focuses on the total weighted completion time objective for which the list of the approximation algorithms is very short.
In this paper we extend that list by considering new special cases and obtain  new complexity results and approximation algorithms.
We show that even if there is only a single non-renewable resource, and each job has unit weight and requires only one unit from the resource, the problem is still NP-hard, however, in our construction we need a high-multiplicity encoding of the jobs in the input.
We also propose  an FPTAS for a  variant in which the jobs have arbitrary weights, and the number of supply time points is bounded by a constant.
Finally, we prove some non-trivial approximation guarantees for  simple greedy algorithms for some further variants of the problem.
\end{abstract}

\section{Introduction}\label{sec:intro}
Single machine scheduling is one of the oldest scheduling problems with many theoretical results and practical applications.
In the recent years the importance of non-renewable resources (like raw materials, energy or money) is increasing.
These resources are consumed by the jobs when the machine starts to process them.
There is an  initial stock, and some additional supplies arrive at given supply dates and in known quantities.
Since the early 80s several papers examined the problem, but mainly  considered "min max" type objectives, such as the makespan, or the maximum lateness.
In this paper we focus on the total (weighted) completion time objective and prove some new complexity and approximability results.

Formally, we have a set of $n$ jobs $\J$ to be scheduled on a single machine, and a non-renewable resource.
Each job $j\in \J$ has a processing time (duration) $p_j>0$, a weight $w_j>0$, and a required quantity $a_j\geq 0$ from the resource. In addition, the number $q\geq 1$ specifies the number of supplies from the non-renewable resource.
Each supply is characterized by a time point $u_\ell\geq 0$ and by a quantity $\tilde{b}_\ell > 0$.  The time points satisfy  $0=u_1<u_2<\ldots<u_q$.
The resource is consumed each time some job $j$ with positive  $a_j$  is started. That is, if some job $j$ is started at time $t$, then the available quantity of the resource must be at least $a_j$, and the inventory of the resource is decreased by $a_j$ immediately at time $t$.
A schedule specifies a start time $S_j$ for each $j \in \J$, and it is {\em feasible\/} if the total supply until any time point $t$ is at least the total demand of those jobs starting not later than $t$. In other words, let $u_\ell$ be the latest supply time point before $t$, then $\sum_{j\in \J:S_j\leq t}a_j\leq \sum _{\ell'=1}^{\ell}\tilde{b}_{\ell'}$ must hold.
We aim at finding a feasible schedule $S$ that minimizes the total weighted completion time $\sum_{j\in\J}w_jC_j$, where $C_j=S_j+p_j$.

Observe that we can assume that the total resource requirement of the jobs matches the total amount supplied, hence at least one job starts not earlier than $u_q$ in any feasible schedule.

Scheduling with non-renewable resources is not only theoretically challenging, but it occurs frequently in practice.
E.g., \cite{Herr16} examine the continuous casting stage of steel production in which hot metal is the non-renewable resource  supplied by  a blast furnace.
A similar problem is studied in \cite{Carrera10} at a shoe-firm and there are examples also in the consumer goods industry and in computer assembly, see \cite{Stadtler08}.
Note that the problem is a special case of the resource constrained project scheduling problem, which has several further practical applications.
We summarize the most important antecedents of this research in Section \ref{sec:prev_work}.

\subsection{Terminology}\label{sec:terminology}

Recall the standard $\alpha|\beta|\gamma$ notation of \cite{Graham79}, where  $\alpha$ indicates the machine environment, the $\beta$ field contains the additional constraints, and the $\gamma$ field provides the objective function. In this notation, our scheduling problem can be compactly represented as $1|nr=1|\sum w_jC_j$, where "$nr=1$" in the $\beta$ field indicates that there is only one type of non-renewable resource, and the "1" in the $\alpha$ field stipulates the single machine environment.

A {\em supply period\/} is a time interval between two consecutive supply time points, and $[u_\ell, u_{\ell+1})$ is the $\ell^{th}$ supply period, where $u_{q+1} = \infty$.
We will  {\em assign\/} jobs to  supply periods, and we say that an assignment is {\em feasible\/} if there is  a schedule in which for each index  $\ell$, the total resource requirements of those jobs assigned to supply periods 1 through $\ell$ does not exceed the total supply over the same periods.

If there are many identical jobs, then the input can be described compactly using a {\em high-multiplicity encoding of the jobs}. Suppose the set of jobs can be partitioned into $h$ classes such that all the jobs in the same class have the same parameters (processing time, job weight, and resource requirement).
Then in the input there is a positive integer number $h$ giving the number of job classes, and for each job class, we have a number $s_i$ providing the number of identical jobs in the class, and 3 other numbers $p_i$, $a_i$ and $w_i$ specifying the common parameters of all the jobs in the class.
If some of these values is the same over all the job classes, then it can be represented only once in the input, but this further simplification does not decrease the size of the input significantly. 
The other input parameters are $q$, the number of supply time points, and the time points $u_\ell$ and supplied quantities $\tilde{b}_\ell$ for $\ell = 1,\ldots,q$.

A polynomial time algorithm on a high-multiplicity input must produce a compact schedule the size of which is bounded by a polynomial in the size of the high-multiplicity input.
A natural schedule representation consists of $h\cdot q$ tuples  $(\ell, i, t_{i\ell}, g_{i\ell})$, where $\ell$ is the index of the supply period, $i$ is that of the job class, $t_{i\ell}$ is the start time of the first job from class $\mathcal{J}_i$ scheduled after  $u_\ell$, and $g_{i\ell}$ is the number of jobs from this class  scheduled consecutively from $t_{i\ell}$ on.
It is easy to see that one can check the feasibility, and compute the objective function value of such a schedule in polynomial time in the size of the high-multiplicity encoded input.
In the  $\alpha|\beta|\gamma$ notation, the tag "$hme$" in the $\beta$ field indicates the high-multiplicity encoding of the input.

A $\rho$-approximation algorithm for our scheduling problem is a polynomial time algorithm that on any input, provides a schedule of objective function value at most $\rho$ times the optimum.
An {\em PTAS\/} for our scheduling problem is a family of algorithms $\{A_\varepsilon\}_{\varepsilon>0}$, that for each $\varepsilon > 0$ contains an algorithm $A_\varepsilon$ which is a factor $(1+\varepsilon)$-approximation algorithm for the problem. An {\em FPTAS} is a family of approximation algorithms with the properties of a PTAS, and in addition, each $A_{\varepsilon}$ has a polynomial time complexity in $1/\varepsilon$ as well.

\subsection{Main results}\label{sec:results}
Firstly, we investigate the complexity and approximability of the problem $1|nr=1, a_j = \bar{a}, q = const.\ |\sum w_j C_j$, i.e., single machine environment,  one non-renewable resource, all jobs have the same required quantity from the resource, which has a constant number of supply time points, and the objective function is the weighted sum of the job completion times.
For our complexity result, we need a high-multiplicity encoding of the input.
\begin{theorem}
The problem $1|nr=1, a_j = 1, q = 2, hme\ |\sum C_j$   is NP-hard.
\label{thm:nph1}
\end{theorem}
That is, minimizing the sum of the job completion times is NP-hard even if there are only two supply time points and all the jobs require only one unit from the resource.
In fact, we reduce the NP-hard EQUAL-CARDINALITY PARTITION problem to our scheduling problem, and we need a huge number of jobs in certain job classes.

For non-constant $q$, and $w_j \equiv 1$, we have a factor 2 approximation algorithm.
\begin{theorem}
Scheduling the jobs in non-decreasing processing time order is a 2-approximation algorithm  for $1|nr=1, a_j=\bar{a}|\sum C_j$.\label{thm:2approx_a=w=1}
\end{theorem}
For $q$ constant, we have stronger results even with arbitrary job weights.
\begin{theorem}
The problem $1|nr=1, a_j = \bar{a}, q = const\ |\sum w_j C_j$  admits an FPTAS.
\label{thm:fptas}
\end{theorem}
When $q=2$, then an FPTAS exists for arbitrary $a_j$ values as well, see \cite{kis15}.
The FPTAS of Theorem~\ref{thm:fptas} can be extended to high-multiplicity encoding of jobs, provided that the number of job classes is bounded by a constant.
\begin{theorem}
The problem $1|nr=1, a_j = \bar{a}, q = const, hme, h=const\ |\sum w_j C_j$  admits an FPTAS.
\label{thm:fptas_hme}
\end{theorem}

The second problem studied in this paper is $1|nr=1, p_j = 1, a_j = w_j \ |\ \sum w_j C_j$, i.e., we have a single machine environment, all jobs have unit processing time, and for each job the  resource requirement equals the weight. This problem has been shown NP-hard in the weak-sense by \cite{Gyorgyi19}.
When the number of supply dates is part of the input, we can prove the following.
\begin{theorem}\label{thm:3approx}
Scheduling the jobs in non-increasing $w_j$ order is a 3-approximation algorithm for $1|nr=1,p_j=1,w_j=a_j|\sum w_j C_j$.
\end{theorem}

However, for $q=2$, more can be said:
\begin{theorem}\label{thm:2approx}
Scheduling the jobs in non-increasing $w_j$ order is a 2-approximation algorithm for $1|nr=1,p_j=1,w_j=a_j,q=2|\sum w_j C_j$.
\end{theorem}

We remark that this theorem remains valid in case of high-multiplicity encoding of the input.

For an overview of previous and new results, see Table~\ref{tab:sumresults}.

\begin{table}[!th]
\caption{Overview of previous and new results for variants of $1|nr=1|\sum w_j C_j$.}
\label{tab:sumresults}
\begin{center}
\begin{small}
\begin{tabular}{|l|l|l|l|}
\hline 
\#\textbf{Supplies} ($q$) &  \textbf{Restriction} & \textbf{Result} & \textbf{Source} \\  
\hline  
\hline
 arbitrary & $p_j =\bar{p},\, a_j = \bar{a}$  & poly.~time algo.  & \cite{Gyorgyi19} \\
 arbitrary & $p_j =\bar{p},\, w_j = \bar{w}$  & poly.~time algo.  & \cite{Gyorgyi19} \\
 arbitrary & $a_j = \bar{a},\, w_j =  p_j$ & poly.~time algo. & \cite{Gyorgyi19}\\
\hline
 arbitrary & $w_j = 1$ & strongly NP-hard & \cite{Carlier84} \\
 2       & $w_j = 1$            & weakly NP-hard      & \cite{kis15}\\
 $2$   & $p_j = 1$, $w_j =  a_j$  & weakly NP-hard & \cite{Gyorgyi19} \\
 $2$   & $w_j = p_j = a_j$  & weakly NP-hard & \cite{Gyorgyi19}\\
 arbitrary & $w_j = p_j = a_j$   & strongly NP-hard & \cite{Gyorgyi19}\\
 $2$ & $a_j = 1$, $hme$ & weakly NP-hard & Theorem~\ref{thm:nph1} of this paper\\
\hline 
arbitrary & $w_j = p_j = a_j$  & 2-approx algo.    & \cite{Gyorgyi19} \\
 constant &  $w_j = p_j$  & PTAS & \cite{Gyorgyi19} \\
 2        &   $-$            & FPTAS & \cite{kis15}\\
 arbitrary &$a_j=\bar{a}$, $w_j = \bar{w}$  & 2-approx algo. & Theorem~\ref{thm:2approx_a=w=1} of this paper\\
 constant &$a_j=\bar{a}$  & FPTAS & Theorem~\ref{thm:fptas} of this paper\\
 constant &$a_j=\bar{a}$, $hme$, $h=const$  & FPTAS & Theorem~\ref{thm:fptas_hme} of this paper\\
arbitrary & $p_j = 1$, $a_j=w_j$  & 3-approx algo.  & Theorem~\ref{thm:3approx} of this paper\\
 2 & $p_j = 1$, $a_j=w_j$  & 2-approx algo.  & Theorem~\ref{thm:2approx} of this paper\\
\hline
\end{tabular}
\end{small}
\end{center}
\end{table}

\subsection{Notation}
\label{sec:notation}
\begin{tabular}{|l|l|}
\hline
$n$ & number of jobs \\
$q$ & number of supply periods\\
$j$ & job index\\
$\ell$ & index of supply \\
$p_j$ & processing time of job $j$\\
$\bar{a}$ & common resource requirement of the jobs \\
$u_\ell$ &  the  $\ell^{th}$ supply time point\\ 
$\tilde{b}_\ell$ & quantity supplied at time point  $u_\ell$\\
$b_\ell$ & total resource supply over the first $\ell$ supplies, i.e.,  $\sum_{k=1}^\ell \tilde{b}_k$\\
$n_\ell$ & total number of jobs that can be served from the first $\ell$ supplies\\
$OPT$ & optimum objective function value of a scheduling problem\\
\hline
\end{tabular}

\medskip

If all the jobs have the same resource requirement $\bar{a}$, we can determine in advance the total number of jobs that can be served from the first $\ell$ supplies. That is, $n_\ell = \lfloor b_{\ell} / \bar{a} \rfloor$. Since $n\bar{a} = \sum_{j=1}^n a_j =   b_q$, we have $n_q = n$.

\section{Previous work}\label{sec:prev_work}
The total weighted completion time objective in a single machine environment without additional resource constraints ($1||\sum w_jC_j$) is solvable in polynomial time, a classical result of \cite{Smith56}.
This objective function is studied in several papers, see e.g., \cite{Hall97}, \cite{Bachman02}, or \cite{Liu19}.

Non-renewable resource constraints in the context of machine scheduling has been introduced by \cite{Carlier84}, and by \cite{Slowinski84}.
For the total weighted completion time objective,  Carlier proved that $1|nr=1|\sum w_jC_j$ is strongly NP-hard.
This result was repeated by \cite{Gafarov11}, who also examined a variant where the supply dates are equidistant and each supplied amount $\tilde{b}_\ell$ is the same.
In \citep{kis15}, it is proved that the problem is still NP-hard (in the weak sense) if there are only two supplies, and also an FPTAS is devised  for this special case.
In a recent paper, \cite{Gyorgyi19} discuss some polynomially solvable special cases of $1|nr=1|\sum w_j C_j$, and identify new NP-hard variants, e.g., when $p_j=w_j=a_j$ for each job $j$.
That paper also describes a 2-approximation algorithm for the above variant, and a PTAS when $p_j = w_j$, and the number of supplies ($q$) is a constant.

There are several papers for the makespan and the maximum lateness objective.
For instance, in \cite{Toker91} it is proved that the single machine makespan minimization problem is equivalent to the two machine flowshop problem if the amount supplied at each time unit is the same, while  in \cite{Gyorgyi17}, the approximability of parallel machine scheduling under non-renewable resource constraints is investigated with  the makespan and the maximum lateness objectives.

According to our best knowledge, high-multiplicity scheduling problems were first examined by \cite{Psaraftis80}, and the term was coined by \cite{Hochbaum91}. 
We also refer to \cite{Grigoriev_phd}, where several high-multiplicity scheduling problems with non-renewable resource constraints are examined, but only for the makespan and for the maximum lateness objectives.

\section{Problem $1|nr=1,a_j=1,q=2, hme|\sum C_j$ is NP-hard}
In this section we prove Theorem~\ref{thm:nph1}.
In  that proof we will use the following lemma several times:

\begin{lemma} \label{lem:obj_hm}
Let $t$ be an arbitrary time point and $S$ an arbitrary feasible schedule of an instance of $1|nr=1, hme|\sum C_j$.
If there are $k$ jobs with the same processing time $p_j$ scheduled without idle time between them from time point $t$ on (in the time interval $\left[t, t+k\cdot p_j\right)$), then their contribution to the objective function value of $S$ is $kt+\binom{k+1}{2}\cdot p_j$.
\end{lemma}
\begin{proof}
Let $j'$ be the job in position $k'$ ($1\leq k'\leq k$) among the jobs specified in the statement of the lemma. 
Since the completion time of $j'$ is $t+k'\cdot p_j$,  the contribution of these jobs to the objective is 
\[
\sum_{k'=1}^k \left(t+k'\cdot p_j\right)= kt + \left(\sum_{k'=1}^k k'\right)\cdot p_j=kt+\binom{k+1}{2}\cdot p_j.
\]\qed
\end{proof}

\begin{proof}[Proof of Theorem~\ref{thm:nph1}]

We reduce the NP-hard EQUAL-CARDINALITY-PARTITION to the scheduling problem in the statement of the theorem. An instance of this problem is characterized by a positive even integer number $n$, and a set of $n$ items  with item sizes $e_1,\ldots,e_n \in \mathbb{Z}_{\geq 0}$ such that $\sum_i e_i=2A$ for some integer $A$, and $e_i\leq 2^{n^2}$ (the last inequality follows from the proof of NP-hardness from \cite{Garey79}).
Question:  is there a subset $H$ of the items such that $|H|=n/2$ and $\sum_{i\in H} e_i=A$?

Let $I$ be an instance of EQUAL-CARDINALITY-PARTITION, we construct an instance $I'$ of the scheduling problem as follows. There are $n'=2\cdot 200^{n^2}+n$ jobs and two supply dates, $u_1=0$ and $u_2=(n/2)\cdot 20^{n^2}+A$ with $b_1=n/2$ and $b_2=2\cdot 200^{n^2} + n/2$, respectively.
$J_1,J_2,\ldots,J_n$ are so called {\em medium\/} jobs with $p_j=20^{n^2}+e_j$ ($j=1,2,\ldots,n$).
We have $200^{n^2}$ small jobs with $p_j=1$, and $200^{n^2}$ big jobs with  $p_j=200^{n^2}$.
Let $\J_s, \J_m$ and $\J_b$ denote the set of small, medium and big jobs, respectively.
The question is if there exists a feasible schedule of total job completion time at most $V' = V_s+V_m+V_b$, where
\begin{align*}
V_s & :=200^{n^2}u_2+\binom{200^{n^2}+1}{2},\\
V_m & :=2\cdot (20^{n^2}+A)\cdot \binom{n/2+1}{2}+(u_2+200^{n^2})\cdot n/2,\\
V_b& :=200^{n^2}\cdot (u_2+200^{n^2}+20^{n^2}\cdot n/2+A)+200^{n^2}\cdot \binom{200^{n^2}+1}{2}\\
 &=200^{n^2}\cdot\sum_{j\in \J_m\cup \J_s}p_j + 200^{n^2}\cdot \binom{200^{n^2}+1}{2}.
\end{align*}

If  the answer to $I$ is "yes", i.e., there is a subset $H$ of $n/2$ items of total size $A$. Then  consider the following schedule for $I'$: schedule the medium jobs corresponding to the elements of $H$ from $0$ to $u_2$ in non-decreasing $p_j$ order, then schedule the remaining jobs in non-decreasing $p_j$ order from $u_2$, i.e., there are small jobs in the time interval $\left[u_2, u_2+200^{n^2}\right]$, medium jobs in the time interval $\left[u_2+200^{n^2}, u_2+200^{n^2}+20^{n^2}\cdot n/2 +A\right]$, and then the big jobs are in the time interval 
\[
\begin{split}
&\left[\sum_{j\in\J_m\cup \J_s}p_j,\sum_{j\in\J_b\cup\J_m\cup \J_s}p_j\right] =\\
&\left[u_2+200^{n^2}+20^{n^2}\cdot n/2+A, u_2+200^{n^2}+20^{n^2}\cdot n/2+A + 200^{2{n^2}}\right], \end{split}
\]
see Figure \ref{fig:np_opt}.
Note that due to Lemma~\ref{lem:obj_hm} the contribution of the small jobs to the objective function value is $V_s$, that of the medium jobs is at most $V_m$, while the contribution of the big jobs is $V_b$.
Therefore, the total job completion time of this schedule is  at most $V'$.

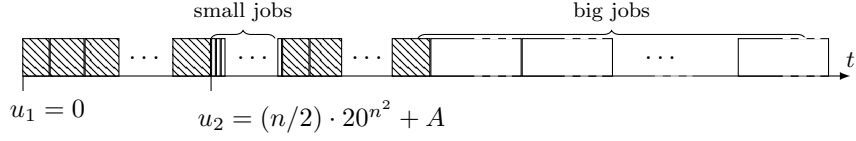
\begin{figure}
\begin{tikzpicture}
\def\ox{0} 
\def\oy{0} 
\def\ui{0}
\def\uii{2.5}
\coordinate(o) at (\ox,\oy); 
\coordinate(u1) at (\ui,\oy);
\coordinate(u2) at (\uii,\oy);

\tikzstyle{mystyle}=[draw, minimum height=0.5cm,rectangle, inner sep=0pt,font=\scriptsize]

\def\tl{11.0} 
\def\oyi{0}
\draw [-latex](\ox,\oyi) node[above left]{} -- (\ox+\tl,\oyi) node[above,font=\small]{$t$};

\coordinate (uq2) at (\uii,\oyi);
\coordinate (uq1) at (\ui,\oyi);
\draw[] (uq2) -- ($(uq2)-(0,0.2)$) node[below right=0cm and -0.3cm] {$u_2=(n/2)\cdot 20^{n^2}+A$}; 
\draw[] (o) -- ($(o)-(0,0.2)$) node[below right=0cm and -0.3cm] {$u_{1}=0$};

\def\pi{0.7}
\node(b1) [above right=-0.01cm and -0.01cm of o,mystyle,pattern=north west lines, minimum width=0.35 cm]{};
\node(b3) [right=0cm of b1,mystyle,pattern=north west lines, minimum width=0.45cm]{};
\node(b3) [right=0cm of b3,mystyle,pattern=north west lines, minimum width=0.45 cm]{};
\node(b3) [right=0.0cm of b3, minimum width=0.1 cm]{$\ldots$};
\node(b3) [above left=-0.01cm and -0.00cm of uq2,mystyle,pattern=north west lines, minimum width=0.5cm]{};

\draw [decorate,decoration={brace,amplitude=3pt},xshift=0pt,yshift=0pt]
(\uii cm,0.5cm) -- (3.35cm,0.5cm) node [above=0.1cm,black,midway]{\footnotesize small jobs};
\node(b2) [right=0.0cm of b3,mystyle, minimum width=0.05 cm]{};
\node(b3) [right=0cm of b2,mystyle, minimum width=0.05 cm]{};
\node(b3) [right=0cm of b3,mystyle, minimum width=0.05 cm]{};
\node(b3) [right=0.03cm of b3, minimum width=0.0 cm]{$\ldots$};
\node(b3) [right=0cm of b3,mystyle, minimum width=0.05 cm]{};
\node(b1) [right=0cm of b3,mystyle,pattern=north west lines, minimum width=0.35 cm]{};
\node(b3) [right=0cm of b1,mystyle,pattern=north west lines, minimum width=0.42cm]{};
\node(b3) [right=0.0cm of b3, minimum width=0.1 cm]{$\ldots$};
\node(b3) [right=0.0cm of b3,mystyle,pattern=north west lines, minimum width=0.5cm]{};
\draw [decorate,decoration={brace,amplitude=3pt},xshift=0pt,yshift=0pt]
(5.25cm,0.5cm) -- (10.4cm,0.5cm) node [above=0.1cm,black,midway]{\footnotesize big jobs};
\node(b1) [right=0cm of b3,mystyle, minimum width=1.2 cm]{};
\node(b3) [right=0cm of b1,mystyle, minimum width=1.2cm]{};
\node(b3) [right=0.3cm of b3, minimum width=0.1 cm]{$\ldots$};
\node(b3) [right=0.7cm of b3,mystyle, minimum width=1.2cm]{};

\draw[dashed, white] (6,0)--(6.5,0);
\draw[dashed, white,thick] (6,0.5)--(6.5,0.5);
\draw[dashed, white] (7.2,0)--(7.7,0);
\draw[dashed, white,thick] (7.2,0.5)--(7.7,0.5);
\draw[dashed, white] (8.4,0)--(8.9,0);
\draw[dashed, white] (10.1,0)--(10.6,0);
\draw[dashed, white,thick] (10.1,0.5)--(10.6,0.5);
\end{tikzpicture}
\caption{If the answer to $I$ is "yes" then the value of $S$ is at most $V'$. The medium jobs are hatched.}\label{fig:np_opt}
\end{figure}

If the answer to $I$  is "no", then we claim that any feasible schedule has a larger objective function value than $V'$.
Let $S^*$ denote an arbitrary optimal schedule.
We distinguish three cases:

\noindent{\em Case 1.\/} There is a big job that starts before $u_2$. 
Then necessarily all the other big jobs must be scheduled after all the small and medium jobs at the end of $S^*$. 
The contribution of the big jobs is at least $200^{n^2}+(200^{n^2}-1)\cdot (200^{n^2}+\sum_{j\in \J_s\cup \J_m}p_j) + 200^{n^2}\cdot\binom{200^{n^2}}{2}
\geq V_b-\sum_{j\in \J_m\cup \J_s}p_j$.

There are at most $n/2-1$ other jobs that start before $u_2$ in $S^*$, thus there are at least $200^{n^2}+(n/2+1)$ small and medium jobs that start after the first big job, i.e., not earlier than  $200^{n^2}$.
The contribution of these jobs to the optimum value is at least $(200^{n^2}+(n/2+1))\cdot 200^{n^2}+\binom{200^{n^2}+(n/2+1)+1}{2}$. 
This is clearly larger than $V_s + V_m + \sum_{j\in \J_m\cup \J_s}p_j$, thus the objective value of $S^*$ is larger than $V'$.
\vskip 5pt

\noindent{\em Case 2.\/} There are $k\geq 1$ small jobs that start before $u_2$, but each big job starts after $u_2$ in $S^*$. Then at least $n/2+k$ medium jobs start after $u_2$ (since at most $n/2$ jobs may start before $u_2$ by the resource constraint), and the machine must be idle in $[u_2-k\cdot 20^{n^2},u_2]$. 
This means that each big job starts not earlier than 
$\sum_{j\in \J_m\cup \J_s}p_j+k\cdot 20^{n^2}$ in $S^*$, because these jobs can start after all small and medium jobs in an optimal schedule.
Hence, the contribution of the big jobs to the objective function value of $S^*$ is at least $V_b+k\cdot 200^{n^2} \cdot 20^{n^2}$.

The contribution of the small and medium jobs is at least $\binom{k+1}{2}+(n/2-k)\cdot k+ \binom{n/2-k+1}{2}\cdot 20^{n^2}+(200^{n^2}-k)u_2+\binom{200^{n^2}-k+1}{2}+(n/2+k)\cdot (u_2+200^{n^2}-k)+\binom{n/2+k+1}{2}\cdot 20^{n^2}$, because in an optimal schedule the jobs are ordered in each period in a non-decreasing $p_j$ order, i.e., both before and after $u_2$, the small jobs precede the medium jobs.
Note that, $\binom{k+1}{2}+(200^{n^2}-k)u_2+\binom{200^{n^2}-k+1}{2}\geq V_s-k\cdot (u_2+200^{n^2})$, while $(n/2-k)\cdot k+ \binom{n/2-k+1}{2}\cdot 20^{n^2}+(n/2+k)\cdot (u_2+200^{n^2}-k)+\binom{n/2+k+1}{2}\cdot 20^{n^2}\geq V_m-n^2A$.
Therefore, the objective function value of $S^*$ is at least $V_b+V_m+V_s+k\cdot 200^{n^2} \cdot 20^{n^2}-n^2A-k\cdot (u_2+ 200^{n^2})>V'$.
\vskip 5pt

\noindent{\em Case 3.\/} 
Only  medium jobs start before $u_2$. Then the remaining jobs are scheduled  in non-decreasing processing time order after the maximum of $u_2$ and the completion of the jobs started before $u_2$.
Since the answer to $I$ is  "no",  either there is a medium job that starts before, but finishes after $u_2$, or there is idle time before $u_2$.
In the former  case, the big jobs  are scheduled in the time interval $\left[\sum_{j\in\J_m\cup \J_s}p_j,\sum_{j\in\J_b\cup\J_m\cup \J_s}p_j\right]$, which means that their contribution to the objective function value is $V_b$.
Observe that in this case exactly $n/2$ medium jobs start before $u_2$, because the total processing time of any $n-1$ medium jobs is less than $u_2$ and there can be at most $n/2$ jobs scheduled before $u_2$ due to the resource constraint of the first period.
The contribution of the medium jobs is at least $2\cdot\binom{n/2+1}{2}\cdot 20^{n^2}+n/2\cdot (u_2+200^{n^2})=V_m-2A\cdot\binom{n/2+1}{2}$, because there are $n/2$ medium jobs that start not earlier than $u_2+200^{n^2}$ (after all small jobs are scheduled after $u_2$). 
The contribution of the small jobs is at least $(200^{n^2}+1)u_2+\binom{200^{n^2}+1}{2}=V_s+200^{n^2}$, because these jobs are scheduled after the last medium job of the first period completes, i.e., not earlier than $200^{n^2}+1$.
Since $200^{n^2}>2A\cdot\binom{n/2+1}{2}$, the value of $S^*$ is larger than $V'=V_b+V_m+V_s$.

Finally, if there is idle time before $u_2$ in $S^*$, then we can suppose that the number of the medium jobs that start before $u_2$ is $n/2$, otherwise the objective value of $S^*$ could be decreased by scheduling a small job before $u_2$, which contradicts the optimality of this schedule.
Since the machine is idle in $[u_2-1,u_2]$, the big jobs cannot start before $\sum_{j\in\J_m\cup \J_s}p_j+1$, therefore their contribution to the objective is at least  $V_b+200^{n^2}$.
The small jobs are scheduled in $\left[u_2,u_2+200^{n^2}\right]$, thus their contribution is exactly $V_s$.
The contribution of the medium jobs is at least $2\cdot (20^{n^2})\cdot \binom{n/2+1}{2}+(u_2+200^{n^2})\cdot n/2=V_m-2A\cdot \binom{n/2+1}{2}$, thus the objective function value of $S^*$ is at least $V_b+V_m+V_s+200^{n^2}-2A\cdot \binom{n/2+1}{2}$, which is clearly larger than $V'=V_b+V_m+V_s$. \qed

\end{proof}

\section{A factor 2 approximation algorithm for $1|nr=1,a_j=\bar{a}|\sum C_j$}

In this section we describe a simple factor 2 approximation algorithm for  the problem $1|nr=1,a_j=\bar{a}|\sum C_j$.
The algorithm is based on list scheduling, i.e., scheduling the jobs in a given order and introducing a gap only if no more material is available when scheduling the next job.

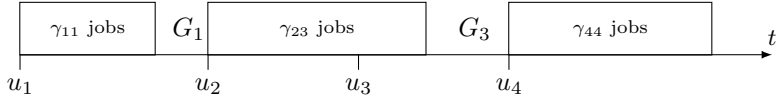
\begin{figure}
\begin{tikzpicture}
\def\ox{0} 
\def\oy{0} 
\def\ui{0}
\def\uii{2.5}
\def\uiii{4.5}
\def\uiv{6.5}
\coordinate(o) at (\ox,\oy); 
\coordinate(u1) at (\ui,\oy);
\coordinate(u2) at (\uii,\oy);
\coordinate(u3) at (\uiii,\oy);
\coordinate(u4) at (\uiv,\oy);

\tikzstyle{mystyle}=[draw, minimum height=0.7cm,rectangle, inner sep=0pt,font=\scriptsize]

\def\tl{10} 
\draw [-latex](\ox,\oy) node[above left]{} -- (\ox+\tl,\oy) node[above,font=\small]{$t$};

\draw[] (o) -- ($(o)-(0,0.2)$) node[below right=0cm and -0.3cm] {$u_{1}$};
\draw[] (u2) -- ($(u2)-(0,0.2)$) node[below right=0cm and -0.3cm] {$u_2$};
\draw[] (u3) -- ($(u3)-(0,0.2)$) node[below right=0cm and -0.3cm] {$u_3$};
\draw[] (u4) -- ($(u4)-(0,0.2)$) node[below right=0cm and -0.3cm] {$u_4$};

\node(b1) [above right=-0.01cm and -0.01cm of o,mystyle, minimum width=1.8 cm]{$\gamma_{11}$ jobs};
\node(b3) [right=0.1cm of b1, minimum width=0.5cm]{$G_1$};
\node(b1) [above right=-0.01cm and -0.01cm of u2,mystyle, minimum width=2.9 cm]{$\gamma_{23}$ jobs};
\node(b3) [right=0.3cm of b1, minimum width=0.5cm]{$G_3$};
\node(b1) [above right=-0.01cm and -0.01cm of u4,mystyle, minimum width=2.7 cm]{$\gamma_{44}$ jobs};
\end{tikzpicture}
\caption{Illustration of list scheduling for $1|nr=1, a_j = \bar{a}|\sum C_j$.}\label{fig:2approx_aw1}
\end{figure}

In order to prove that the list-scheduling algorithm using the non-decreasing processing-time order is a factor 2 approximation algorithm we need some additional definitions.
Recall the definition of $n_\ell$ from Section~\ref{sec:notation}.
Let $\nu_{k\ell}$ denote the difference $n_\ell-n_k$ for $0\leq k\leq \ell\leq q$, where $n_0 := 0$.
Let $SPT$ denote the optimum value of the problem $1||\sum C_j$, i.e., minimize the sum of job completion times on a single machine. $SPT$ can be obtained by scheduling the jobs in non-decreasing processing time order without any delays, a classical result of  \cite{Smith56}.
\begin{lemma}
For any problem instance of $1|nr=1,a_j=\bar{a}|\sum C_j$, the optimum  value $OPT$  admits the following lower bounds:
\begin{enumerate}[(i)]
\item $SPT \leq OPT$, \label{case1}
\item $\sum_{\ell=2}^q u_\ell\cdot \nu_{\ell-1,\ell} \leq OPT$. \label{case2}
\end{enumerate}
\end{lemma}
\begin{proof}
Since $1||\sum C_j$ is a relaxation of $1|nr=1,a_j=\bar{a}|\sum C_j$, (\ref{case1}) follows immediately.
As for (\ref{case2}), fix an instance of $1|nr=1,a_j=\bar{a}|\sum C_j$ and consider an optimal schedule for that instance. Let $C^*_j$ denote the completion time of job $j$ in the optimal schedule. We can express each $C^*_j$ as the sum of a supply time point $u_{\ell(j)} \in \{u_1,\ldots,u_q\}$ and some integer number $v_j$ such that $C^*_j = u_{\ell(j)} + v_j$ and $u_{\ell(j)} \leq C^*_j - p_j < u_{\ell(j)+1}$ (i.e., job $j$ starts in supply period $\ell$).
Let $\gamma^*_\ell$ be the number of those jobs that start in supply period $\ell$ in the optimal schedule.
Then we have
\begin{equation}
OPT = \sum_{j=1}^n C^*_j = \sum_{j=1}^n (u_{\ell(j)} + v_j) > \sum_{j=1}^n u_{\ell(j)} = \sum_{\ell = 1}^q u_\ell\cdot \gamma^*_\ell \geq \sum_{\ell=1}^q u_\ell\cdot \nu_{\ell-1,\ell},\label{eq:lb}
\end{equation}
where only the last inequality needs justification.
Consider a relaxed problem, where all the job processing times are set to 0.
Then the optimal schedule for $1|nr=1,a_j=\bar{a}|\sum C_j$ is just a feasible schedule for this relaxed problem with objective function value $\sum_{\ell = 1}^q u_\ell\cdot \gamma^*_\ell$, whereas in the optimal solution, $\nu_{\ell-1,\ell}$ jobs are scheduled at time point $u_\ell$.
This immediately gives the last inequality in (\ref{eq:lb}).
\qed
\end{proof}

\begin{proof}[of Theorem \ref{thm:2approx_a=w=1}]
The main idea of the proof is that we express the objective function value of the schedule $S_{Alg}$ obtained by the algorithm using the gaps in the schedule and then we upper bound it appropriately.
Let $G_\ell$ be the idle time or gap in the interval $[u_\ell, u_{\ell+1})$ in the schedule $S_{Alg}$, and $C_j$ the completion time of job $j$. Let $\gamma_{k\ell}$ denote the number of jobs that start in the interval $[u_k,u_{\ell+1})$, see Figure~\ref{fig:2approx_aw1}. Furthermore, let $\{\ell_1,\ldots,\ell_t\}$ be the set of supply period indices such that $G_{\ell_i-1} > 0$ (there is a gap in the schedule in supply period $\ell-1$), and suppose $\ell_i < \ell_{i+1}$ for $1\leq i< t$. If this set is not empty, then  $\ell_1 \geq 2$ must hold by definition.
We also define $\ell_{t+1} := q+1$.
Then we have
\[\begin{split}
\sum_{j=1}^n C_j & = SPT + \sum_{i=1}^t  \left(\sum_{k=1}^{\ell_i} G_k\right) \cdot \gamma_{\ell_i,\ell_{i+1}-1} \leq SPT + \sum_{i=1}^t  u_{\ell_i} \cdot \gamma_{\ell_i,\ell_{i+1}-1}\\
& \leq SPT + \sum_{i=1}^t  u_{\ell_i} \cdot \nu_{\ell_i-1,\ell_{i+1}-1} \leq SPT + \sum_{\ell=1}^t u_{\ell}\cdot \nu_{\ell-1, \ell} \leq 2\cdot OPT,
\end{split}
\]
where the first equation is based on the same idea as we used in the previous lemma, the second inequality uses the obvious fact that the total idle time before $u_\ell$ is at most $u_\ell$, the third inequality exploits that for each $i$, there is a gap right before $u_{\ell_i}$, and before $u_{\ell_{i+1}}$ as well, the fourth from $\nu_{\ell_i-1, \ell_{i+1}-1} = \sum_{k=\ell_i}^{\ell_{i+1}-1} \nu_{k-1,k}$, and the last from the previous lemma. \qed
\end{proof}

Finally, we mention that the above list scheduling algorithm can be extended the $hme$-input and the same approximation guarantee can be proved.

\section{FPTAS for $1|nr=1, a_j = \bar{a}, q=const, hme, h=const|\sum w_j C_j$}

\subsection{FPTAS for normal input }

Firstly, we describe a dynamic program, and then we sketch how to turn it into an FPTAS.
We assume that the jobs are indexed in non-increasing $w_j / p_j$  order, i.e., $w_1/p_1 \geq \cdots \geq w_n / p_n$.

Our dynamic program is defined by an acyclic graph, where the nodes represent states, and the edges the transitions between them. Each state $\sigma$ is a $4q$ tuple $(N_1^\sigma,\ldots,N_q^\sigma; P_1^\sigma,\ldots,P_q^\sigma; W_1^\sigma,\ldots,W_q^\sigma; WP^\sigma_1,\ldots, WP^ \sigma_q) \in \mathbb{R}^{4q}$, where $N_\ell^\sigma$, $P_\ell^\sigma$, $W_\ell^\sigma$, and $WP_\ell^\sigma$ represent the total number, the total processing time, the total weight, and the total weighted completion time  (if started at time 0) of those jobs assigned to the supply period $\ell$. 
Note that if $j_1,\ldots,j_k$ are the jobs assigned to supply period $\ell$ in state $\sigma$ such that $j_1 < j_2 < \cdots < j_k$, then $WP_{\ell}^\sigma = \sum_{i=1}^k w_{j_i} \left(\sum_{s=1}^i p_{j_s}\right)$. The initial state is the all 0 vector.
Consider any state $\sigma = (N_1^\sigma,\ldots,N_q^\sigma; P_1^\sigma,\ldots,P_q^\sigma; W_1^\sigma,\ldots,W_q^\sigma; WP^\sigma_1,\ldots,WP^\sigma_q)$ with $\sum_{\ell'=1}^\ell N_{\ell'}^\sigma \leq n_\ell$ for each $\ell \in \{1,\ldots,q\}$, and with strict inequality for at least one $\ell$. Let $j := \sum_{\ell=1}^q N_{\ell}^\sigma + 1$ the index of the next job to be scheduled. For each $\ell$ such that $\sum_{\ell'=1}^\ell N_{\ell'}^\sigma < n_{\ell}$, we define a successor state $\sigma'$, unless it is already defined, as follows.
In $\sigma'$, the values of $N_\ell^{\sigma'}$, $P_\ell^{\sigma'}$, $W_\ell^{\sigma'}$ and $WP_\ell^{\sigma'}$ are computed as $N_\ell^{\sigma} + 1$, $P_\ell^{\sigma} + p_j$, $WP^\sigma_\ell + w_j$, and
$WP_{\ell}^{\sigma} + (P^{\sigma}_\ell+p_j)\cdot w_j$, respectively; while all other components are inherited from $\sigma$.
The {\em terminal\/} states are those $\sigma$ with $\sum_{\ell=1}^q N_{\ell}^\sigma = n$.
The objective function value of a terminal state is computed as
\begin{equation}
\operatorname{value}(\sigma) := \sum_{\ell=1}^q \left(WP_\ell^\sigma + \max\left\{u_\ell, \max_{\ell'<\ell}\left(u_{\ell'}+\sum_{k=\ell'}^{\ell-1} P_k^\sigma\right) \right\}\cdot W_\ell^\sigma\right).
\label{eq:statevalue}
\end{equation}
Note that in the above expression, $WP_\ell^\sigma + \max\left\{u_\ell, \max_{\ell'<\ell}u_{\ell'}+\sum_{k=\ell'}^{\ell-1} P_k^\sigma \right\}\cdot W_\ell^\sigma$ expresses the total weighted completion time of those jobs assigned to supply period $\ell$. To see this, observe that in $\sigma$, the first job in supply period $\ell$ starts at $t_\ell := \max\left\{u_\ell, \max_{\ell'<\ell}\left(u_{\ell'}+\sum_{k=\ell'}^{\ell-1} P_k^\sigma \right) \right\}$. Since the total weight of those jobs assigned to supply period $\ell$ is $W_\ell^\sigma$, $WP^\sigma_\ell$ must be increased by $t_\ell \cdot W^\sigma_\ell$ to get the total weighted completion time of these jobs, and this yields the formula (\ref{eq:statevalue}).

We determine all the terminal states, and choose the best one, i.e., with the smallest value.
Note that resource feasibility is ensured by the definition of $n_\ell$ and the fact that $\sum_{k=1}^\ell N_{k}^\sigma \leq n_\ell$ in each (terminal) state $\sigma$.

We claim that the running time of this procedure is pseudo polynomial. To see this, notice that any number in any state can be bounded by $(n\cdot \operatorname{MAXNUM})^2$, where $\operatorname{MAXNUM}$ is the maximum number in the input. Since $q$ is a constant, the number of states can be bounded by $\operatorname{SOL}(n,q)\cdot (n\operatorname{MAXNUM})^{5q}$, where $\operatorname{SOL}(n,q)$ is the total number of solutions of the Diophantine equation system 
$\sum_{k=1}^\ell N_k \leq n_\ell$, $\ell=1,\ldots,q$. This can be bounded by $n^q$. Therefore, the running time can be bounded by a polynomial in $n$ and $\operatorname{MAXNUM}$, if $q$ is a constant.
Therefore, we have proved the following:
\begin{lemma}
The problem $1|nr=1,q=const, a_j = \bar{a}| \sum w_j C_j$ can be solved in pseudo-polynomial time.
\end{lemma}

Luckily, this pseudo-polynomial time algorithm can be turned into an FPTAS under the same conditions.

Let $\Delta = 1+\varepsilon/(2n)$. We shall use the following rounding function:
\[
r(v) =\left\{\begin{array}{ll}
0, & \textnormal{if } v = 0\\
\Delta^{\lceil \log_{\Delta} v\rceil}, & \textnormal{if } v > 0.
\end{array}
\right.
\]
A notable property of this function is that if $v_1,\ldots, v_t$ is a sequence of $t \leq n$ non-negative numbers, and $g_i = r(v_i + g_{i-1})$, where $g_0 = 0$, then
\begin{align}
\sum_{j=1}^i v_j \leq g_i \leq (1+\varepsilon)\sum_{j=1}^i v_j,\quad i=1,\ldots,t.\label{eq:r_bound}
\end{align}
The first inequality follows from $g(v) \geq v$, and the second from  $(1 + \alpha/n)^n \leq e^\alpha\leq 1+2\alpha$ for $0\leq \alpha < 1$, see \cite{schuurman2007approximation}.
In the following algorithm we modify the above dynamic program by rounding the states.
A state is {\em non-terminal\/} if at least one job is not assigned to a supply period in it.
Consider the following algorithm:
\begin{enumerate}[1.]
\item The initial state is the $4q$-dimensional 0 vector. 
The successors of a non-terminal state $\sigma = (N_1^{\sigma},\dots,N_q^{\sigma}; \tilde{P}_1^{\sigma},\dots,\tilde{P}_q^{\sigma}; \tilde{W}_1^{\sigma},\dots,\tilde{W}_q^{\sigma}; \tilde{WP}^\sigma_1,\dots,\tilde{WP}^\sigma_q)$ are computed as follows. If $\sum_{\ell'=1}^{\ell} N_{\ell'}^{\sigma} < n_\ell$, then job $j= 1+\sum_{\ell'=1}^q N_{\ell'}^{\sigma}$ can be assigned to supply period  $\ell$.
The components of the corresponding state $\sigma'$ are inherited from $\sigma$, except 
 $N_\ell^{\sigma'}$, $\tilde{P}_\ell^{\sigma'}$, $\tilde{W}_\ell^{\sigma'}$, and $\tilde{WP}_\ell^{\sigma'}$, which are computed as  $N_\ell^{\sigma}+1$, $r(\tilde{P}_\ell^{\sigma}+p_j)$, $r(\tilde{W}_\ell^{\sigma}+w_j)$ and $r(\tilde{WP}_\ell^{\sigma}+w_j (\tilde{P}^{\sigma}+p_j))$, respectively.
\item After computing the (rounded) terminal states of this dynamic program, we take a terminal state $\sigma'$ of smallest value, and pick any path from the initial state leading to $\sigma'$. By following this path, a solution to the scheduling problem is  constructed, and this is the output of the algorithm. 
\end{enumerate}
Firstly, we observe that the output of the algorithm is a feasible schedule and it has a value at most  that of $\sigma'$ by the properties of the function $r(\cdot)$. Therefore, in order to show that the above algorithm  constitutes an FPTAS for the scheduling program, we have to show that for any $\varepsilon>0$, $\operatorname{value}(\sigma')$ is at most $(1+O(\varepsilon))$ times the optimum, and also provide a polynomial bound on the time complexity in terms of the size of the scheduling problem instance and $1/\varepsilon$.

Now consider the original (unrounded) dynamic program and a path $\Pi$ from the initial state to an optimal terminal state $\sigma^*$ (i.e., $\operatorname{value}(\sigma^*)$ takes the optimum value).
Clearly, the assignment of jobs to supply periods can be read out from $\Pi$, and a terminal state $\tilde{\sigma}$ of the rounded dynamic program is reached by making these assignments in the algorithm above. 
Since $\tilde{P}^{\tilde{\sigma}}_\ell \leq (1+\varepsilon) P^*_\ell$, $\tilde{W}^{\tilde{\sigma}}_\ell \leq (1+\varepsilon) W^*_\ell$, and $\tilde{WP}^{\tilde{\sigma}}_\ell \leq (1+\varepsilon) WP^*_\ell$ by (\ref{eq:r_bound}), we can bound the value of $\tilde{\sigma}$ as follows:
\[
\begin{split}
\operatorname{value}(\tilde{\sigma}) & \leq (1+\varepsilon)^2\sum_{\ell=1}^q \left(WP_\ell^{\sigma^*} + \max\left\{u_\ell, \max_{\ell'<\ell}\left(u_{\ell'}+\sum_{k=\ell'}^{\ell-1} P_k^{\sigma^*}\right) \right\}\cdot W_\ell^{\sigma^*}\right)\\
& < (1+3\varepsilon)\operatorname{value}(\sigma^*),
\end{split}
\]
that is, the value of $\tilde{\sigma}$ is at most $(1+3\varepsilon)$ times the optimum.
Since we pick the best solution of the rounded dynamic program, the output of the algorithm has the same approximation guarantee.

It remains to verify the time complexity of the dynamic program with the rounded states.
The crucial factor in determining the running time is the number of distinct values of the components $\tilde{P}_\ell$, $\tilde{W}_\ell$ and $\tilde{WP}_\ell$ of the rounded states.
Since $\log_\Delta \sum p_j = \ln \sum p_j / \ln \Delta \leq 4n \ln \sum p_j / \varepsilon$ for any $0 < \varepsilon < 1$, the number of distinct $\tilde{P}_\ell$ values is bounded by a polynomial $poly(|I|,1/\varepsilon)$ in the size of the input and in $1/\varepsilon$. A similar bound can be given for the number of distinct $\tilde{W}_\ell$ values, while the number of distinct $\tilde{WP}_\ell$ can be bounded by $\log_\Delta (\sum p_j)(\sum w_j)$, which is also bounded by polynomial in the input size and in  $1/\varepsilon$. Therefore, the time complexity of the algorithm on any input $I$ with $n$ jobs is $O(n^q poly(|I|,1/\varepsilon) ^{3q})$. 
Hence, we proved Theorem~\ref{thm:fptas}.

\subsection{FPTAS for {\em hme\/}-input}

In this section we describe how to modify the FPTAS of the previous section to deal with {\em hme\/}-input.
Recall that in such an input, there are $h$ job classes and the number of jobs in class $\mathcal{J}_i$ is $s_i$.
Since $\sum_{i=1}^h s_i$ may not be polynomially bounded in the size of the {\em hme\/}-input, the FPTAS of the previous section would have a pseudo-polynomial time complexity if applied directly to {\em hme\/}-input.

Firstly, we need a slightly different rounding function. Let $\bar{\Delta} := (1+\varepsilon/(2h))$, and
\[
\bar{r}(v) := \left\{\begin{array}{ll}
0, & \textnormal{ if } v = 0,\\
\bar{\Delta}^{\lceil \log_{\bar{\Delta}} v \rceil}, & \textnormal{ if } v > 0.
\end{array}
\right.
\]
Recall that the FPTAS of the previous section runs in $n$-stages ($n$ is the number of the jobs), and in stage $j$, job $j$ is assigned to one of the $\ell$ supply periods. We modify this strategy as follows. In the FPTAS for {\em hme}-input, there are $h$ stages, one for each job class. We assume that $w_1/p_1 \geq \cdots \geq w_h/p_h$, where $w_i$ and $p_i$ are the common weight and processing time, respectively, of all the jobs of job class $i$. The states of the new dynamic program are labelled by $4q+1$ tuples of the form $(i; N_1^\sigma,\ldots,N_q^\sigma; \tilde{P}_1^\sigma,\ldots,\tilde{P}_q^\sigma; \tilde{W}_1^\sigma,\ldots,\tilde{W}_q^\sigma; \tilde{WP}_1^\sigma,\ldots,\tilde{WP}_q^\sigma)$, and they  differ from the states of the dynamic program of the previous section in one important aspect:  the first component is the index of the job class scheduled last.
The initial state is the all-zero vector.
Each arc connects a state $\sigma$ at some stage $i-1\geq 0$ with a state $\sigma'$ at stage $i\leq h$, and it is labelled with a tuple $(\delta_{i1},\ldots,\delta_{iq})$, where $\sum_{\ell=1}^q \delta_{i\ell} = s_i$, which provides the number of jobs from class $\mathcal{J}_i$ assigned to each of the supply periods.
Since the number of such tuples is in the order of $\Omega(s_i^q)$, which is not bounded by a polynomial in the size of the {\em hme}-input in general, we cannot enumerate all the possible assignments of the jobs from class $\mathcal{J}_i$ to supply periods in an algorithm of polynomial running time in the size of the {\em hme}-input and in $1/\varepsilon$.
For this reason, consider the quantities $(1+\varepsilon)^k$, for $k \in \mathcal{K}_i := \{z \in \mathbb{Z}\ |\ 0\leq z \leq \lceil \log_{(1+\varepsilon)} s_i \rceil\}$.
We will use the tuples $(k_1,\ldots, k_q)$, where each $k_\ell \in \mathcal{K}_i$.
Let $E_i$ be the set of eligible tuples, where a tuple  $(k_1,\ldots, k_q)$ is {\em eligible\/} if and only if $s_i \leq \sum_{\ell=1}^q \lfloor (1+\varepsilon)^{k_\ell}\rfloor$.
The number of eligible tuples is bounded by $|\mathcal{K}_i|^q$, which is bounded by $((2\ln s_i) / \varepsilon)^q$, since $|\mathcal{K}_i| \leq (2\ln s_i) / \varepsilon$ as a standard computation shows. However, this is polynomially bounded in the size of the  {\em hme\/}-input and in $1/\varepsilon$.
For each tuple $(k_1,\ldots, k_q) \in E_i$, we compute an assignment of jobs to supply periods by the following Allocation algorithm:
\begin{enumerate}[1.]
\item Let $t = s_i$, and $\ell = q$.
\item While $\ell > 0$ do
\item \ \ Let  $\delta_{i\ell} := \min\{t, \lfloor(1+\varepsilon)^{k_\ell}\rfloor\}$, and $t := t-\delta_{i\ell}$ \label{algo:calc_delta}
\item \ \ Let $\ell := \ell-1$
\item End do 
\item Output: $(\delta_{i1},\ldots,\delta_{iq})$.
\end{enumerate}
Clearly, the output of the algorithm satisfies $\sum_{\ell=1}^q \delta_{i\ell} = s_i$ provided that $(k_1,\ldots,k_q) \in E_i$.
Observe that the jobs of class $\mathcal{J}_i$ are assigned backward, from supply period $q$ to supply period 1. The use of this allocation strategy is in the proof of feasibility of the set of tuples corresponding to the optimal solution, as we will see later.

Consider any state $\sigma = (i-1;  N_1^\sigma,\ldots,N_q^\sigma; \tilde{P}_1^\sigma,\ldots,\tilde{P}_q^\sigma; \tilde{W}_1^\sigma,\ldots,\tilde{W}_q^\sigma; \tilde{WP}_1^\sigma,$ $\ldots,\tilde{WP}_q^\sigma)$.
For each distinct tuple $(\delta_{i1},\ldots,\delta_{iq})$, a subsequent state $\sigma'$ of $\sigma$ is defined as $\sigma' = (i; N_1^{\sigma'},\ldots, N_q^{\sigma'}; \tilde{P}_1^{\sigma'}, \ldots,\tilde{P}_q^{\sigma'}; \tilde{W}_1^{\sigma'},\ldots,\tilde{W}_q^{\sigma'}; \tilde{WP}_1^{\sigma'},\ldots,\tilde{WP}_q^{\sigma'})$, where
$N^{\sigma'}_\ell := N_\ell^\sigma + \delta_{i\ell}$,
$\tilde{P}_\ell^{\sigma'} := \bar{r}( \tilde{P}_\ell^\sigma + \delta_{i\ell} \cdot p_i)$, 
$\tilde{W}_\ell^{\sigma'} := \bar{r}(\tilde{W}_\ell^{\sigma} + \delta_{i\ell} \cdot w_i)$, and
$\tilde{WP}_\ell^{\sigma'} := \bar{r}(\tilde{WP}_\ell^{\sigma} + (\delta_{i\ell}\cdot \tilde{P}_\ell^\sigma + \delta_{i\ell} \cdot (\delta_{i\ell}+1)/2 \cdot p_i)\cdot w_i)$.
If $\sigma'$ is already stored at stage $i$, then the processing of $\sigma'$ is finished. 
Otherwise, we check the {\em feasibility\/} of $\sigma'$ by verifying the condition
\begin{equation}
\sum_{\ell'=1}^\ell N_{\ell'}^{\sigma'} \leq  n_\ell, \textnormal{ for } \ell = 1,\ldots, q-1.\label{eq:feasy}
\end{equation}
We store  $\sigma'$ at stage $i$ only if it  satisfies (\ref{eq:feasy}).

 The states obtained  at stage  $h$ are the {\em terminal states\/}. Clearly, all terminal states represent feasible 
 allocation of jobs to supply periods.
Among the terminal states, we pick the one with smallest value, computed by the formula (\ref{eq:statevalue}).
The solution is obtained by repeatedly moving to the predecessor states until the initial state is reached. The arcs visited provide an eligible assignment from each $E_i$, that together determine a solution of the problem. 

It remains to verify the approximation ratio and the time complexity of the algorithm.
Consider an optimal solution of the scheduling problem, and for each job class $\mathcal{J}_i$ and supply period $\ell$,
let $n^*_{i\ell}$ denote the number of jobs from class $\mathcal{J}_i$ started in the interval $[u_\ell, u_{\ell+1})$, if $\ell < q$, and not before $u_q$ if $\ell = q$. Let $k_{i\ell}$ be the smallest integer such that $n^*_{i\ell} \leq \lfloor (1+\varepsilon)^{k_{i\ell}}\rfloor$.  Clearly, all the tuples $(k_{i1} ,\ldots, k_{iq})$ are eligible.
Let $(\delta_{i1},\ldots,\delta_{iq})$ be the job allocation returned by the Allocation algorithm for the tuple $(k_{i1} ,\ldots, k_{iq})$ for $i=1,\ldots,q$. Consider the sequence of states $\sigma_1$, $\sigma_2,\ldots,\sigma_h$ obtained by  applying the job allocations $(\delta_{i1},\ldots,\delta_{iq})$ in increasing order of the index $i$.

\begin{claim}
The states $\sigma_1,\ldots,\sigma_h$ satisfy the condition (\ref{eq:feasy}), and for each $i$, $\sigma_i$ is a state stored at stage $i$ of the algorithm.
\end{claim}
\begin{proof}
Clearly, the algorithm will generate $\sigma_1$. If it satisfies the condition (\ref{eq:feasy}), then it will generate $\sigma_2$ from it, etc. It suffices to prove that $\sigma_h$ satisfies the condition (\ref{eq:feasy}), because it implies that all previous states do.
By the rules of the Allocation algorithm, $\sum_{\ell'=\ell}^q \delta_{i\ell'} \geq \sum_{\ell'=\ell}^q n^*_{i\ell'}$ for each $\ell$ and $i$. Consequently, $\sum_{\ell'=1}^\ell \delta_{i\ell'} \leq \sum_{\ell'=1}^\ell n^*_{i\ell'}$, since $s_i = \sum_{\ell'=1}^q \delta_{i\ell'} = \sum_{\ell'=1}^q n^*_{i\ell'}$. Since the optimal solution is feasible, we have 
\[
\sum_{\ell' =1}^\ell N^{\sigma_h}_{\ell'} =  \sum_{i=1}^h \sum_{\ell=1}^\ell \delta_{i\ell'} \leq  \sum_{i=1}^h\sum_{\ell'=1}^\ell n^*_{i\ell'} \leq n_\ell,\textnormal{ for } \ell=1,\ldots,q-1,
\]
which proves our claim. \qed
\end{proof}
Let $\sigma$ be the state that is obtained from the initial state by applying  the job allocations $(\delta_{i1},\ldots,\delta_{iq})$ in increasing order of the index $i$, but without rounding the components $P_\ell$, $W_\ell$ and $WP_\ell$ by $\bar{r}(\cdot)$.  Using $\sigma$, the value of $\sigma_h$ can be bounded as follows:
\begin{equation}
\operatorname{value}(\sigma_h) < (1+\varepsilon)^2\sum_{\ell=1}^q \left(WP_\ell^\sigma + \max\left\{u_\ell, \max_{\ell'<\ell}\left(u_{\ell'}+\sum_{k=\ell'}^{\ell-1} P_k^\sigma\right) \right\}\cdot W_\ell^\sigma\right),
\label{eq:bound_sigma_h}
\end{equation}
where the inequality follows from the properties of the rounding function $\bar{r}(\cdot)$.
We have to relate the right-hand-side of the above expression to the value of the optimal solution.
Let $P_\ell^*$, $W^*_\ell$ and $WP_\ell^*$ denote  the total processing time, the total weight, and the total weighted completion time  (if started at time 0) of those jobs assigned to supply period $\ell$ in the optimal solution.
Notice that  $\delta_{i\ell} \leq \lfloor (1+\varepsilon)^{k_{i\ell}}\rfloor \leq (1+\varepsilon) n^*_{i\ell}$.
It follows that
\[
\begin{split}
P_\ell^\sigma & = \sum_{i=1}^h \delta_{i\ell} \cdot  p_i \leq (1+\varepsilon) \sum_{i=1}^h n^*_{i\ell} \cdot  p_i = (1+\varepsilon) P^*_\ell, \\
W_\ell^\sigma & = \sum_{i=1}^h \delta_{i\ell} \cdot  w_i  \leq (1+\varepsilon) \sum_{i=1}^h n^*_{i\ell} \cdot w_i = (1+\varepsilon) W^*_\ell, \textnormal{ and }\\
WP_\ell^\sigma & = \sum_{i=1}^h \left( \delta_{i\ell} \cdot \left(\sum_{j=1}^{i-1} \delta_{j\ell} \cdot p_j \right) + p_i \cdot \delta_{i\ell}\cdot (\delta_{i\ell} +1)/2 \right) \cdot w_i  \\
& \leq (1+\varepsilon)^2\sum_{i=1}^h \left( n^*_{i\ell} \cdot \left(\sum_{j=1}^{i-1} n^*_{j\ell} \cdot p_j \right) + p_i \cdot n^*_{i\ell}\cdot (n^*_{i\ell} +1)/2 \right) \cdot w_i \\
	& =  (1+\varepsilon)^2 WP_\ell^*.
\end{split}
\]
Since the optimum value can be expressed as
\begin{equation*}
OPT = \sum_{\ell=1}^q \left(WP_\ell^* + \max\left\{u_\ell, \max_{\ell'<\ell}\left(u_{\ell'}+\sum_{k=\ell'}^{\ell-1} P_k^*\right) \right\}\cdot W_\ell^*\right),
\end{equation*}
the right-hind-side of (\ref{eq:bound_sigma_h}) can be bounded by $(1+\varepsilon)^4 \cdot OPT$, hence,
the value of $\sigma_h$ is at most $(1+\varepsilon)^4 \cdot OPT$. Since the algorithm chooses the terminal state with smallest value, and $\sigma_h$ is one of the terminal states, the value of the best terminal state is at most  $(1+\varepsilon)^4 \cdot OPT$, which is $(1+O(\varepsilon))\cdot OPT$, since we can assume that $0< \varepsilon \leq 1$.

Finally, the time complexity of the algorithm is proportional to the number of distinct $N_1,\ldots,N_q$ values that can be obtained by choosing an eligible tuple from each $E_i$. This can be bounded by $O(\Pi_{i=1}^h(2(\ln s_i)/\varepsilon)^q)$, which is bounded by $O((2/\varepsilon)^{q\cdot h} \cdot (\Pi_{i=1}^ h \ln   s_i ))^q)$, a polynomial in the size of the {\em hme}-input and in $1/\varepsilon$, provided that $q$ and $h$ are constants.
Therefore, Theorem~\ref{thm:fptas_hme} is proved.

\section{Approximation  of $1|nr=1,p_j=1,w_j=a_j|\sum w_j C_j$}

In this section first we prove Theorem \ref{thm:3approx}, and then  Theorem \ref{thm:2approx}.
For the sake of simpler notation, we assume that the jobs are indexed in non-increasing $w_j$ order, i.e.,  $w_1\geq w_2 \geq \ldots \geq w_n$.

\begin{proof}[Proof of Theorem~\ref{thm:3approx}]
Let $S$ be the solution obtained by scheduling the jobs in non-increasing $w_j$ order as early as possible while respecting the resource constraint, and $S^*$ an optimal schedule.
Let $W_\ell$ and $W^*_\ell$ be the total weight of the jobs that start in $[u_\ell,u_{\ell+1})$ in $S$, and   in $S^*$, respectively.
For $\ell=1,\dots,q-1$, let $k_{\ell}$ be the index of the last job that starts before $u_{\ell+1}$ in $S$.
Let $G_\ell$ and $G^*_\ell$ denote the length of the idle period in $[u_\ell,u_{\ell+1})$ in $S$  and  in $S^*$, respectively.
Let $s_\ell:=G_\ell-G^*_\ell$.
Since $w_j = a_j$ for all jobs, and job $k_\ell+1$ is started not sooner that $u_{\ell+1}$ in $S$, 
we have 
\[
\sum_{\ell'=1}^\ell W_{\ell'}+w_{k_\ell+1} > b_{\ell} \geq \sum_{\ell'=1}^\ell W^*_{\ell'}, \qquad  \ell=1,\dots,q-1,
\]
thus $\sum_{\ell'=\ell+1}^q W_{\ell'}<\sum_{\ell'=\ell+1}^q W^*_{\ell'}+w_{k_\ell+1}$ for $\ell=1,\dots,q-1$.

Note that both of $G_\ell$ and $G^*_\ell$ are at most $u_{\ell+1}-u_\ell$.

Since the jobs are scheduled in non-increasing $w_j$ order in schedule $S$, the objective function value of this schedule is:
\begin{align}
\sum_{j=1}^n jw_j+\sum_{\ell=1}^{q-1} G_\ell \left(\sum_{\ell'=\ell+1}^q W_{\ell'}\right)\leq  \sum_{j=1}^n jw_j+\sum_{\ell=1}^{q-1} G_\ell \left( \sum_{\ell'=\ell+1}^q W^*_{\ell'}+w_{k_\ell+1} \right)\label{eq:obj_S}
\end{align}
On the other hand, we can bound the optimum value from below as follows:
\begin{align*}
\sum_{j=1}^n jw_j+\sum_{\ell=1}^{q-1} G^*_\ell \left(\sum_{\ell'=\ell+1}^q W^*_{\ell'}\right) \leq OPT.
\end{align*}
Hence, the difference between the value of $S$ and the optimum can be bounded from above by
\begin{align}
\sum_{\ell=1}^{q-1} s_\ell \left(\sum_{\ell'=\ell+1}^q W^*_{\ell'} \right)+ \sum_{\ell=1}^{q-1} w_{k_\ell+1}G_\ell.\label{eq:difference}
\end{align}
The first part of the above expression is  at most the optimum value, because $s_\ell\leq u_{\ell+1} - u_\ell$, and then
\[
\begin{split}
\sum_{\ell=1}^{q-1} s_\ell \left(\sum_{\ell'=\ell+1}^q W^*_{\ell'} \right) & \leq \sum_{\ell=1}^{q-1} (u_{\ell+1}-u_\ell) \left(\sum_{\ell'=\ell+1}^q W^*_{\ell'} \right)\\
& = \sum_{\ell=1}^q W^*_\ell \sum_{\ell'=1}^{\ell-1} (u_{\ell'+1}-u_{\ell'}) = \sum_{\ell=1}^q W^*_\ell u_\ell \leq OPT.
\end{split}
\]
The second term of (\ref{eq:difference}) can be bounded by $\sum_{\ell=1}^{q-1} w_{k_\ell+1}(u_{\ell+1}-u_\ell)$, since $G_\ell \leq u_{\ell+1} - u_\ell$. It remains to bound this latter term.

Since in $S$ the jobs are scheduled in non-increasing $w_j$ order, there is a job $j_1\leq k_1+1$ that starts after $u_2$ in $S^*$.
Suppose that it starts in $[u_{\ell_1},u_{\ell_1+1})$.
It contributes to the optimum by at least $w_{k_1+1}u_{\ell_1}$, which is at least $\sum_{\ell=1}^{\ell_1-1} w_{k_\ell+1}(u_{\ell+1}-u_\ell)$.
Furthermore, if $\ell_1 < q$, there is a job $j_2\leq k_{\ell_1}+1$ that starts after $u_{(\ell_1+1)}$ in $S^*$.
Suppose  $j_2$ starts in $[u_{\ell_2},u_{\ell_2+1})$, thus it contributes to the optimum by at least $w_{k_{\ell_1}+1}u_{\ell_2}\geq \sum_{\ell=\ell_1}^{\ell_2-1} w_{k_{\ell}+1}(u_{\ell+1}-u_\ell)$.
We can continue this until we encounter a job $j_t$ that must start after $u_q$ in schedule $S^*$. 
Consequently,
\[
\sum_{\ell=1}^{q-1} w_{k_\ell+1}G_\ell \leq \sum_{\ell=1}^{q-1} w_{k_\ell+1}(u_{\ell+1}-u_\ell) \leq w_{k_1+1} u_{\ell_1} + \sum_{i=2}^t w_{k_{\ell_{i-1}+1}} u_{\ell_i} \leq OPT. 
\]
Thus the second term of (\ref{eq:difference}) is also at most the optimum, hence (\ref{eq:difference}) is at most two times the optimum, therefore $S$ has an objective function value of at most $3\cdot OPT$. \qed
\end{proof}

\begin{proof}[Proof of Theorem~\ref{thm:2approx}]
Let $S$ be the schedule found by the algorithm and $S^*$ an optimal schedule.
We use the same notation as in the proof of Theorem \ref{thm:3approx}, but for simplicity we introduce $G^*:=G^*_1$.
Note that if $G^*=0$, then the algorithm yields an optimal schedule.
For the sake of a contradiction, suppose that there is an instance where the theorem is not true.
Consider a counterexample $I$ with minimal number of jobs, i.e., $\sum_{j=1}^n w_j C_j > 2 \sum_{j=1}^n w_j C^*_j$, where $C_j = S_j +1$ and $C^*_j = S^*_j +1$.

\begin{claim}
Job $J_1$ starts at $u_2$ in $S^*$, i.e., $S^*_1 = u_2$.
\end{claim}
\begin{proof}
Since $J_1$ has the largest weight, if $J_1$ is not started at $u_2$ in the optimal schedule, then it must be started at time 0, i.e., $S^*_1=0$. Then consider the instance $I'$ obtained from $I$ by dropping $J_1$ and by decreasing $b_1$ by $w_1$ and $u_2$ by 1. Then the algorithm gives a schedule $S'$ such that $S'_j = S_j-1$ for each $j =2,\ldots,n$. Furthermore, the objective function value of $S'$ is related to that of $S$ as follows:
\begin{equation}
\sum_{j=2}^n w_j C'_j  = \sum_{j=1}^n w_j(C_j-1)  = \sum_{j=1}^n w_j C_j  - \sum_{j=1}^n w_j.\label{eq:objS'}
\end{equation}
On the other hand, we can derive a new feasible schedule for $I'$ from $S^*$.
Let $\tilde{S}_j = S^*_j-1$ for $j=2,\ldots,n$. This schedule is again feasible, and its value is
\begin{equation}
\sum_{j=2}^n w_j \tilde{C}_j  = \sum_{j=1}^n w_j (C^*_j-1)  = \sum_{j=1}^n w_j C^*_j  - \sum_{j=1}^n w_j.\label{eq:objSt}
\end{equation}
Comparing  (\ref{eq:objS'}), and (\ref{eq:objSt}), we get that $I'$ is also a counterexample with fewer jobs than $I$, a contradiction. \qed
\end{proof}
From now on we assume that $S^*_1 = u_2$.

Let $J_k$ be the last job scheduled before $u_2$ in $S$, see Figure \ref{fig:2approx_q2}.
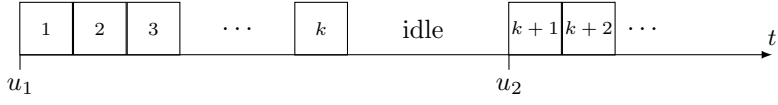
\begin{figure}
\begin{tikzpicture}
\def\ox{0} 
\def\oy{0} 
\def\ui{0}
\def\uii{6.5}
\coordinate(o) at (\ox,\oy); 
\coordinate(u1) at (\ui,\oy);
\coordinate(u2) at (\uii,\oy);

\tikzstyle{mystyle}=[draw, minimum height=0.7cm,rectangle, inner sep=0pt,font=\scriptsize]
\def\pi{0.7}

\def\tl{10} 
\def\oyi{0}
\draw [-latex](\ox,\oyi) node[above left]{} -- (\ox+\tl,\oyi) node[above,font=\small]{$t$};

\coordinate (uq2) at (\uii,\oyi);
\coordinate (uq1) at (\ui,\oyi);
\draw[] (uq2) -- ($(uq2)-(0,0.2)$) node[below right=0cm and -0.3cm] {$u_2$}; 
\draw[] (o) -- ($(o)-(0,0.2)$) node[below right=0cm and -0.3cm] {$u_{1}$};

\node(b1) [above right=-0.01cm and -0.01cm of o,mystyle, minimum width=\pi cm]{$1$};
\node(b3) [right=0cm of b1,mystyle, minimum width=\pi cm]{$2$};
\node(b3) [right=0cm of b3,mystyle, minimum width=\pi cm]{$3$};
\node(b3) [right=0cm of b3, minimum width=1.5 cm]{$\ldots$};
\node(b3) [right=0cm of b3,mystyle, minimum width=\pi cm]{$k$};
\node(b3) [right=0.6cm of b3, minimum width=\pi cm]{idle};

\node(b2) [above right=-0.01cm and -0.01cm of uq2,mystyle, minimum width=\pi cm]{$k+1$};
\node(b3) [right=0cm of b2,mystyle, minimum width=\pi cm]{$k+2$};
\node(b4) [right=0cm of b3, minimum width=\pi cm]{$\ldots$};

\end{tikzpicture}
\caption{Schedule $S$, where the jobs are in non-increasing $w_j$ order.}\label{fig:2approx_q2}
\end{figure}
We can describe the objective function value of $S$ as a special case of (\ref{eq:obj_S}), but now we choose a slightly different form for technical reasons: 
\begin{align}
\sum_{j=1}^n w_jC_j=\sum_{j=1}^k j w_j+u_2\sum_{j=k+1}^n w_j+\sum_{j=k+1}^n (j-k)w_j=\notag\\
\sum_{j=1}^k j w_j+(u_2-k+1)\sum_{j=k+1}^n  w_j+    \sum_{j=k+1}^n (j-1)w_j.\label{eq:2approx_S}
\end{align}

We also give a new expression for the objective function of the optimum schedule $S^*$. Let $\pi^*$ be bijection between the set of positions $\{1,\ldots,n\}$, and the set of jobs such that $\pi^*(i) = j$ if job $j$ is in position $i$ of the optimal schedule $S^*$. Then we have
\begin{align}
\sum_{j=1}^n w_j C^*_j & = \sum_{i=1}^n i\cdot  w_{\pi^*(i)} + W^*_2 \cdot G^* \notag\\
& = (u_2+1)w_1 + \sum_{i=2}^n (i-1)\cdot  w_{\pi^*(i)} + (W^*_2-w_1) (G^*+1)\notag\\
& \geq  (u_2+1)w_1+\sum_{j=2}^n (j-1)w_j+(W^*_2-w_1)\cdot (G^*+1),\label{eq:2approx_opt}
\end{align}
where the inequality follows from the fact that $w_2 \geq w_3 \geq \cdots \geq w_n$, and thus
the sum $\sum_{i=2}^n (i-1)\cdot  w_{\pi^*(i)}$ is minimized by the permutation which assigns job $j$ to position $j$.
The difference of (\ref{eq:2approx_S}) and (\ref{eq:2approx_opt}) is 
\begin{align}
&\sum_{j=1}^k j w_j+(u_2-k+1)\sum_{j=k+1}^n w_j-(u_2+1)w_1-\sum_{j=2}^k(j-1)w_j-(W^*_2-w_1)G^*-\notag\\
&W^*_2+w_1=\sum_{j=1}^k w_j+(u_2-k+1)(W^*_1+W^*_2-\sum_{j=1}^k w_j)-
u_2w_1-(W^*_2-w_1)G^*-W^*_2.\label{eq:diff_q2}
\end{align}

We have to prove that (\ref{eq:diff_q2}) cannot be larger than the optimum.
Since $G^*\neq 0$ and $w_j=a_j$ for all job $j$, $\sum_{j=1}^{k+1}w_j>b_1$ follows, because otherwise the algorithm could have scheduled  job $k+1$ earlier.
However, $W^*_1\leq b_1$, because $S^*$ is feasible, thus we have $W^*_1<\sum_{j=1}^{k+1}w_j$ and therefore the difference is at most 
\begin{align}
&\sum_{j=1}^k w_j+(u_2-k+1)(W^*_2+w_{k+1})-u_2w_1-(W^*_2-w_1)G^*-W^*_2=\notag\\
&\sum_{j=1}^k w_j+u_2W^*_2+u_2(w_{k+1}-w_1)-kW^*_2-(k-1)w_{k+1}-(W^*_2-w_1)G^*.\label{eq:diff_q2_2}
\end{align}
Now, if $k=0$, then then (\ref{eq:diff_q2_2}) simplifies to
\[
u_2 W_2^* + w_1 - (W^*_2 - w_1)G^*.
\]
However, this last expression is a lower bound on the optimum value, since the contribution of those jobs that start after $u_2$ in $S^*$ is at least  $(u_2+1)W^*_2$
and the largest-weight job starts at $u_2$ in $S^*$ as well.

Finally, suppose that $k \geq 1$. Then  (\ref{eq:diff_q2_2}) can be bounded from above by
\[
\sum_{j=1}^k w_j+u_2W^*_2,
\]
because $W^*_2\geq w_1\geq w_{k+1}$.
Furthermore, $\sum_{j=1}^n w_jC^*_j\geq \sum_{j=1}^k w_j+u_2W^*_2$, because each  $C^*_j\geq 1$ and there are jobs with a total weight of at least $W^*_2$ with a completion time of at least $u_2+1$ in $S^*$, thus the theorem follows. \qed
\end{proof}

We have a tight example for this case.
Consider an instance where we have only 2 jobs: $j_1$ with weight $w$ and $j_2$ with weight $w-\varepsilon$. 
Let $\tilde{b}_1:=w-\varepsilon$, $\tilde{b}_2:=w$ and $u_2:=w$.
The algorithm schedules $j_1$ from $u_2$, and $j_2$ from $u_2+1$, thus the objective function value of the resulting schedule is 
\[
(u_2+1)\cdot w+ (u_2+2)\cdot (w-\varepsilon)= 2w^2+3w-\varepsilon(w+2).
\]
However, we can schedule $j_2$ from $t=0$ and $j_1$ from $u_2$ and the value of the resulting schedule is 
\[
w-\varepsilon+ (u_2+1)\cdot w=w^2+2w-\varepsilon.
\]
Note that the relative error of the algorithm on this instance is $\left(\frac{2w^2+O(w)}{w^2+O(w)}\right)$, which tends to 2 as $w$ goes to infinity.

\section{Conclusion}
We have shown several approximation results for different variants of $1|nr=1|\sum w_jC_j$. 
However, there are still a lot of open problems in this area. 
For instance, it is unknown whether there is a polynomial time constant factor approximation algorithm for $1|nr=1|\sum w_jC_j$ or not.
We have conjectured that scheduling the jobs in non-increasing $w_j$ order is a factor 2 approximation algorithm for  $1|nr=1,p_j=1,w_j=a_j|\sum w_jC_j$, but until now we could not prove it.

\section*{Acknowledgement}
This work was supported by the National Research, Development and Innovation Office -- NKFIH, Grant no.~SNN~129178, and ED\_18-2-2018-0006. 
The research of P\'eter Gy\"orgyi was supported by the J\'anos Bolyai Research Scholarship of the Hungarian Academy of Sciences.

\bibliographystyle{apacite}
\bibliography{mybibfile}

\end{document}